\documentclass[11pt,oneside,reqno]{amsart}

\usepackage[OT2,T1]{fontenc}
\usepackage[utf8]{inputenc}
\usepackage{amssymb}
\usepackage{bm,bbm}
\usepackage{mathrsfs}

\usepackage[dvipsnames]{xcolor}
\usepackage[
colorlinks=true,
linkcolor=Maroon,
citecolor=JungleGreen,
urlcolor=NavyBlue]{hyperref}

\usepackage{xcolor,colortbl}
\usepackage[capitalize,nameinlink]{cleveref}

\usepackage{enumitem}
\usepackage{dsfont}
\usepackage{tikz}
\usetikzlibrary{patterns}
\usetikzlibrary{arrows}
\usepackage{multirow}

\usepackage{lipsum}

\usepackage[normalem]{ulem}

\usepackage{upgreek,float}

\newtheorem{theorem}{Theorem}[section]

\newtheorem{lemma}[theorem]{Lemma}
\newtheorem{proposition}[theorem]{Proposition}

\newtheorem{remark}[theorem]{Remark}

\numberwithin{equation}{section}

\newcommand{\PZ}{\mathbb{Z}_+}

\newcommand{\R}{\mathbb{R}}
\newcommand{\e}{\mathrm{e}}

\newcommand{\fl}{\mathfrak{l}}
\newcommand{\fg}{\mathfrak{g}}
\newcommand{\Id}{\mathrm{Id}}
\newcommand{\idct}{\mathbf{1}}
\newcommand{\cH}{\mathcal{H}}

\newcommand{\arcosh}{\mathop{\mathrm{arcosh}}}

\newcommand{\0}{\mathbf{0}}
\newcommand{\bc}{\mathbf{c}}
\newcommand{\F}{\mathscr{F}}
\renewcommand{\C}{\mathbb{C}}
\newcommand{\Z}{\mathbb{Z}}
\renewcommand{\Re}{\mathop{\mathrm{Re}}}
\renewcommand{\SS}{\mathbb{S}}
\renewcommand{\H}{\mathbb{H}}
\renewcommand{\i}{\mathrm{i}}
\renewcommand{\d}{\mathrm{d}}
\renewcommand{\L}{\mathcal{L}_{*}}
\newcommand{\SO}{\mathrm{SO}}

\newcommand{\df}[1]{\,\textrm{d}#1}

\makeatletter
\def\@tocline#1#2#3#4#5#6#7{\relax
\ifnum #1>\c@tocdepth 
\else
\par \addpenalty\@secpenalty\addvspace{#2}%
\begingroup \hyphenpenalty\@M
\@ifempty{#4}{%
\@tempdima\csname r@tocindent\number#1\endcsname\relax
}{%
\@tempdima#4\relax
}%
\parindent\z@ \leftskip#3\relax
\advance\leftskip\@tempdima\relax
\rightskip\@pnumwidth plus4em \parfillskip-\@pnumwidth
#5\leavevmode\hskip-\@tempdima
\ifcase #1
\or\or \hskip 2em \or \hskip 2em \else \hskip 3em \fi%
#6\nobreak\relax
\dotfill\hbox to\@pnumwidth{\@tocpagenum{#7}}\par
\nobreak
\endgroup
\fi}
\makeatother

\begin{document}
\title[Lacunary Hyperbolic Spherical Maximal Operators]{Lacunary Spherical Maximal Operators \\ on Hyperbolic Spaces}

\author[Y. Wang, H.-W. Zhang]{Yunxiang Wang, Hong-Wei Zhang}

\begin{abstract}
We prove that the lacunary spherical maximal operator, defined on the $n$-dimensional real hyperbolic space, is bounded on $L^p(\H^n)$ for all $n\ge2$ and $1<p\le\infty$. In particular, the lacunary set is significantly larger than its Euclidean counterpart, reflecting the influence of the geometry at infinity of the hyperbolic space.
\end{abstract}


\makeatletter
\@namedef{subjclassname@2020}{\textnormal{2020}
\it{Mathematics Subject Classification}}
\makeatother

\maketitle
\section{Introduction}
The study of maximal operators has long been central to harmonic analysis. The Hardy-Littlewood maximal operator $M_*$, defined as the supremum of the averages of a locally integrable function $f$ over balls $B(x,t)$ centered at a given point $x\in\R^n$:
\begin{align*}
M_{*}(f)(x) 
= \sup_{t > 0} \frac{1}{|B(x,t)|} 
\int_{B(x,t)} |f(y)| \df{y},
\end{align*}
plays a crucial role in establishing pointwise convergence results. It is well-known that $M_{*}$ is bounded on $L^p(\R^n)$ for all $p>1$.

In the 1970s, Stein introduced the spherical maximal operator, which deals with spheres of varying radius instead of averaging over balls:
\begin{align}\label{Eq:1.001}
S_* (f)(x) 
= \sup_{t>0}
|(f*\!\df{\sigma_t})(x)|
= \sup_{t>0} 
\left|\int_{\SS^{n-1}}f(x-ty) \df{\sigma(y)}\right|.
\end{align}
Here $\df{\sigma}$ and $\df{\sigma_t}$ are the normalized surface measures of the unit sphere $\SS^{n-1}$ and the sphere of radius $t$, respectively. Such an operator plays a key role in the pointwise convergence problem of wave propagators. Unlike the Hardy-Littlewood maximal operator, the spherical maximal operator $S_*$ is $L^p$-bounded only for $p>n/(n-1)$, and this threshold is sharp. Such a result was proved by Stein \cite{Ste76} for $n\ge3$,  and the case of $n=2$ was later completed by Bourgain \cite{Bou86}.

Since then, the study of variants of maximal operators and spherical maximal operators averaging over different geometric settings has attracted much interest. See, for instance, \cite{Cor77,CW78,SW78,Cal79,Bou91,Sog91,SW11,PR15,Lac19,LWZ23,CK24,KO24}. Among these contributions, the lacunary spherical maximal operator has gained considerable attention for its refined $L^p$-boundedness properties.

The \textit{lacunary spherical maximal operator}, introduced by Coifman and Weiss \cite{CW78} and Calder\'on \cite{Cal79}, takes the supremum of spherical averages over the lacunary set $\Lambda_{\R}=\lbrace{2^{-j}\mid j\in\Z_+}\rbrace$:
\begin{align}\label{def: lac Rn}
S_{*}^{\Lambda_{\R}} (f)(x)
= \sup_{t\in\Lambda_{\R}}
|(f*\!\df{\sigma_t})(x)|
= \sup_{j\in\Z_+} 
\left|\int_{\SS^{n-1}}f(x-2^{-j}y) \df{\sigma(y)}\right|.
\end{align}
This lacunary version is  $L^p$-bounded for all  $p>1$, which improves the above $L^p$-boundedness results of the spherical maximum operator $S_*$.

Beyond the Euclidean setting, maximal operators have been extensively studied on Heisenberg groups and more general two-step nilpotent Lie groups. See, for instance, \cite{NT97,MS04,NT04,Ver06,GT21,RSS23}. In particular, it was recently shown by Ganguly and Thangavelu \cite{GT21} that the lacunary spherical maximal operator, with a lacunary set similar to $\Lambda_{\R}$, is $L^p$-bounded for all $p>1$. Recall that dilation structures can be defined in these settings.

We are interested in maximal operators defined on the real hyperbolic space $\H^n$ ($n\ge2$), an $n$-dimensional Riemannian manifold with sectional curvature $-1$ that grows exponentially fast at infinity. Therefore, balls in $\H^n$ do not satisfy the doubling property, and certain Euclidean techniques may fail in this setting. Despite the significantly different geometric properties, it has been shown that maximal operators defined on $\H^n$ exhibit similar $L^p$-boundedness behavior to those on $\R^n$: Clerc and Stein \cite{CS74} and Str\"{o}mberg \cite{Str81} demonstrated that the Hardy-Littlewood maximal operator is bounded from $L^p(\H^n)$ to $L^p(\H^n)$ for all $p > 1$ and from $L^1(\H^n)$ to $L^{1,\infty}(\H^n)$, respectively; later, El Kohen \cite{ElK80} and Ionescu \cite{Ion00} showed that the spherical maximal operator is bounded from $L^p(\H^n)$ to $L^p(\H^n)$ for all $p>n/(n-1)$. See also \cite{NS97,Nev98,MNS00}, where the maximal operator is studied from the perspective of ergodic theory in more general symmetric spaces.

This paper aims to define the lacunary counterpart to the spherical maximal operator on $\H^n$ and to study its $L^p$-boundedness. To begin with, we define the \emph{lacunary set} in $\H^n$
\begin{align}\label{eq:deflac}
\Lambda=\lbrace{2^{-j}\mid j\in\Z_+}\rbrace\cup\Z_{+}.
\end{align}
Note that $\Lambda$ is significantly larger than $\Lambda_\R$ defined previously (see Figure \ref{fig:comlac} for comparison). Moreover, in the Euclidean setting, the maximal operator associated with the lacunary set $\Lambda$ is no longer $L^p$-bounded for $1<p \le n(n-1)$, see Remark \ref{rmk: 1.2}. This difference reflects the influence of the geometry at infinity of $\mathbb{H}^n$, which yields exponential decay in the estimates for both the kernel and the Fourier multiplier.

\begin{figure}
\centering

\begin{tikzpicture}
\begin{scope}[xshift=0cm]
\fill (0.05,0) circle (0.2mm);
\fill (0.1,0) circle (0.2mm);
\fill (0.15,0) circle (0.2mm);
\draw[line width=0.3mm] (0.2,0) -- (8.5,0);
\fill (8.6,0) circle (0.2mm);
\fill (8.7,0) circle (0.2mm);
\fill (8.8,0) circle (0.2mm);
\draw[line width=0.3mm] (0,-0.1) node[below] {$0$} -- (0,0.1);
\fill (0.25,0) node[below] {$\frac14$} circle (0.5mm);
\fill (0.5,0) node[below] {$\frac12$} circle (0.5mm);
\fill (1,0) node[below] {$1$} circle (0.5mm);
\fill (2,0) node[below] {$2$} circle (0.5mm);
\fill (3,0) node[below] {$3$} circle (0.5mm);
\fill (4,0) node[below] {$4$} circle (0.5mm);
\fill (5,0) node[below] {$5$} circle (0.5mm);
\fill (6,0) node[below] {$6$} circle (0.5mm);
\fill (7,0) node[below] {$7$} circle (0.5mm);
\fill (8,0) node[below] {$8$} circle (0.5mm);
\node[right] at (9,0) {$\Lambda$};
\end{scope}

\begin{scope}[yshift=-1cm]
\fill (0.05,0) circle (0.2mm);
\fill (0.1,0) circle (0.2mm);
\fill (0.15,0) circle (0.2mm);
\draw[line width=0.3mm] (0.2,0) -- (8.5,0);
\fill (8.6,0) circle (0.2mm);
\fill (8.7,0) circle (0.2mm);
\fill (8.8,0) circle (0.2mm);
\draw[line width=0.3mm] (0,-0.1) node[below] {$0$} -- (0,0.1);
\fill (0.25,0) node[below] {$\frac14$} circle (0.5mm);
\fill (0.5,0) node[below] {$\frac12$} circle (0.5mm);
\fill (1,0) node[below] {$1$} circle (0.5mm);
\fill (2,0) node[below] {$2$} circle (0.5mm);
\fill (4,0) node[below] {$4$} circle (0.5mm);
\fill (8,0) node[below] {$8$} circle (0.5mm);
\node[right] at (9,0) {$\Lambda_{\R}$};
\end{scope}

\end{tikzpicture}
\caption{Comparison between $\Lambda$ and $\Lambda_{\R}$.}
\label{fig:comlac}
\end{figure}

The main object of this paper is the \emph{lacunary spherical maximal operator $\L$ on $\H^n$}, defined by
\begin{align}\label{def: lac}
\L (f)(x)
= \sup_{t\in\Lambda} |(f*\!\df{\omega_t})(x)|,
\end{align}
where $\df{\omega}_t$ be the normalized spherical measure of $\H^n$. See Section \ref{Sec:2} for a detailed explanation of these notations and a more precise expression of the operator $\L$. Our main result is as follows.

\begin{theorem}\label{thm: main}
Let $\H^n$ be a real hyperbolic space of dimension $n\ge2$. Then the lacunary spherical maximal operator $\L$, defined by \eqref{def: lac}, is bounded from $L^p(\H^n)$ to $L^p(\H^n)$ for all $1<p\le\infty$.
\end{theorem}

\begin{remark}\label{rmk: 1.2}
In view of the results on the full spherical maximal operator obtained in \cite{ElK80,Ion00}, Theorem \ref{thm: main} shows that the lacunary spherical maximal operator $\L$ enjoys broader $L^p$-boundedness due to its lacunary structure. This phenomenon is similar to that in the Euclidean setting.

However, as mentioned above, the lacunary set $\Lambda$ in $\H^n$ is much larger than $\Lambda_{\R}$, defined for the operator \eqref{def: lac Rn} in $\R^n$. In fact, if we replace the lacunary set $\Lambda_{\R}$ by $\Lambda$, the operator
\begin{align*}
    S_{*}^{\Lambda}=\sup_{t\in\Lambda}\left|\int_{\SS^{n-1}}f(x-ty)\,\d\sigma(y)\right|
\end{align*}
is no longer bounded on $L^{p}(\R^n)$ for $1<p\le n/(n-1)$. To see this, we note that it follows from a sacling argument and Fatou's lemma that the operator $S_{*}^{\Lambda}$ has the same $L^p$-boundedness as the operator
\begin{align*}
    S_{*}^{\Lambda'}=\sup_{t\in{\Lambda'}}\left|\int_{\SS^{n-1}}f(x-ty)\,\d\sigma(y)\right|
\end{align*}
with $\Lambda'=\cup_{k=1}^{\infty}\lbrace{2^{-k}\lambda \mid \lambda\in\Lambda}\rbrace$, see also the remark right after \cite[Corollary 6.3.3]{Sog17}. Since $\Lambda'$ is dense in $\R_{+}$, the operator $S_{*}^{\Lambda'}$ in turn has the same $L^p$-boundedness as the operator $S_{*}$ defined by \eqref{Eq:1.001}, and is therefore bounded on $L^{p}(\R^n)$ if and only if $p>n/(n-1)$.

\end{remark}

Theorem \ref{thm: main} follows from the study of a large family of maximal operators $\lbrace{\L^{\alpha}\mid\alpha\in\C,\ \Re\alpha>(1-n)/2}\rbrace$ satisfying $\L^{0}=\L$, see \eqref{ennn} for its definition in terms of Fourier multipliers and Theorem \ref{thm: main alpha} for its $L^p$-boundedness result. Compared to the previously known results on spherical maximal operators obtained in \cite{ElK80,CSWY24}, the lacunary maximal operators $\L^{\alpha}$ possess broader $L^p$-boundedness, see Remark \ref{rmk: range}.

Moreover, some of the main ingredients in the proof of Theorem \ref{thm: main}, such as the Kunze–Stein phenomenon, also hold on more general symmetric spaces of non-compact type. However, another key ingredient, namely the explicit expression for the corresponding Fourier multiplier, is currently available only in the hyperbolic setting. It would therefore be both interesting and meaningful to investigate whether the enlarged lacunary structure persists in greater generality, and to determine its sharp range in that setting.

The organization of this paper is straightforward: after reviewing the analysis on hyperbolic spaces in Section \ref{Sec:2}, we establish the $L^p$-boundedness of the operators $\L^{\alpha}$ in Section \ref{Sec:3}. We decompose the operator $\L^{\alpha}$ into its \textit{local} and \textit{nonlocal} parts and establish their $L^p$-boundedness separately. The local part is handled using a Calder\'on-type argument from \cite{Cal79}, adapted to the hyperbolic space, while the nonlocal part relies on the Kunze-Stein phenomenon. In both cases, we make use of estimates for the multipliers associated with $\L^{\alpha}$. Throughout this paper, we use the notation $f \lesssim g$ for two positive functions to indicate that there exists a constant $C > 0$, which may depend on the parameters $\alpha$, $p$, or $n$ but is independent of the variables of both functions, such that $f \le C g$.

\section{Preliminaries}\label{Sec:2}
In this section, we review the analysis on the hyperbolic space. For $n\ge2$, we represent the $n$-dimensional real hyperbolic space $\H^n$ using the hyperboloid model:
\begin{align*}
\H^n=\{x=(x_0,x_1,\dots,x_n)\in\R^{1+n} \mid [x,x]=1,\ x_0\geq 1\},
\end{align*}
where $[x,y]=x_0y_0-x_1y_1 - \cdots - x_{n}y_{n}$ denotes the Lorentz quadratic form on $\H^n$. We denote by $[x]=[x,x]$ for all $x\in\mathbb{H}^n$. The Riemannian distance between two arbitrary points on $\H^n$ is then given by
\begin{align*}
d(x,y) = \arcosh([x,y]).
\end{align*}
One can also express $\H^n$ using general polar coordinates as
\begin{align*}
\H^n=\{(x_0,x')\in\R^{1+n} \mid
(x_0,x')=(\cosh r, \sigma \sinh r),\ r\ge0,\ \sigma\in\SS^{n-1}\}.
\end{align*}
In this context, the distance from a point to the origin $\0=(1,0,\dots,0)$ of $\H^n$ is given by
\begin{align*}
|x| = d(x, \0) 
= d((\cosh r, \sigma \sinh r), \0) = r.
\end{align*}

It is also known that the real hyperbolic space is a non-compact symmetric space of rank $1$. Then $\H^n$ can be realized as the homogeneous space $G/K$, where $G= \SO_{\mathrm{e}}(n,1)$ and $K=\SO(n)$. Recall that $K$ is the maximal compact subgroup of the non-compact semisimple Lie group $G$, and $G$ can be decomposed via the \textit{Cartan decomposition} $G=K\overline{A_{+}}K$, where $\overline{A_{+}}$ denotes the closure of $A_{+}$ which consists of matrices
\begin{align*}
a(r)=\left(\begin{array}{ccc}
\cosh r & \sinh r & 0\\
\sinh r & \cosh r & 0\\
0 & 0 & \Id_{n-1}
\end{array}\right),
\quad r\ge0.
\end{align*}
According to the Cartan decomposition, we denote by $\tau: ka(r) \mapsto ka(r)\cdot\0$ the diffeomorphism from $K\overline{A_{+}}$ onto $\H^n$.

Let $\!\df{\omega_r}$ be the normalized spherical measure such that 
\begin{align*}
\int_{\H^n}f \df{\omega_r}=\int_K f(k a(r)\cdot\0)\,\df{k},
\end{align*}
where $\!\df{k}$ is the normalized Haar measure on the compact subgroup $K$ such that $\int_{K}\df{k}=1$. Therefore, the lacunary spherical maximal operator \eqref{def: lac} is given by 
\begin{align*}
\L (f)(x)
= \sup_{t\in\Lambda} |(f*\!\df{\omega_t})(x)|
= \sup_{t\in\Lambda} 
\left| \int_K f(\tau^{-1}(x)ka(t)\cdot\0) \df{k}\right|.
\end{align*}

\subsection{Harmonic analysis on the hyperbolic space}
For $\sigma\in\SS^{n-1}$, we denote by $b(\sigma)=(1,\sigma)$ a point in $\R^{1+n}$. For any $f\in C^\infty_c(\H^n)$, the Fourier transform is defined by
\begin{align}\label{def: Fourier}
(\F f)(\lambda,\sigma) = 
\int_{\H^n}f(x)\,[x,b(\sigma)]^{\i\lambda-\frac{n-1}{2}} \df{x},
\end{align}
see, for instance, \cite[(1.3)]{Bra94}.

A function $f$ is said to be radial on $\H^n$, or bi-$K$-invariant on $G$, if it depends only on the distance $r$ of $x$ to the origin $\0$. For such a function, one can express it in the polar coordinates:
\begin{align}\label{Eq:2.1}
\int_{\H^n}f(x)\,\d x=\sigma_{n-1}\int_0^\infty f(r)\,\sinh^{n-1}\!r\,\d r,
\end{align}
where $\sigma_{n-1}$ is the surface area of the unit sphere $\SS^{n-1}$.

For a radial function, the Fourier transform \eqref{def: Fourier} reduces to the following spherical Fourier transform (or Harish-Chandra transform):
\begin{align*}
\F f(\lambda)= \sigma_{n-1} 
\int_0^\infty f(r)\, \varphi_{\lambda}(r)\,\sinh^{n-1}\!r\,\d r,
\end{align*}
where $\varphi_{\lambda}$ denotes the elementary spherical function given by
\begin{align}\label{Eq:2.2}
\varphi_\lambda(r)
= 
{\Gamma(n/2)\over \sqrt{\uppi}\,
\Gamma((n-1)/2)}\int_0^\uppi (\cosh r-\cos s\,\sinh r)^{\i\lambda-\frac{n-1}{2}}
\sin^{n-2}\!s\,\d s,
\end{align}
see \cite[(3.1.2) and (3.3.2)]{Bra94}. It is well known that the spherical function $\varphi_\lambda$ can be defined in several equivalent ways, for instance, using the Jacobi function, the Gauss hypergeometric function, or, in particular, the Legendre function:
\begin{align}\label{def: phi P}
\varphi_\lambda(r)
= 2^{\frac{n}{2}-1} \Gamma\!\left(\frac{n}{2}\right)
(\sinh r)^{-\frac{n}{2}+1}
P_{-\frac{1}{2}+\i\lambda}^{-\frac{n}{2}+1} (\cosh r),
\end{align}
where the Legendre function $P^\mu_\nu$ has the integral representation
\begin{align*}
P^\mu_\nu(\zeta)=\frac{2^{-\nu}(\zeta^2-1)^{-\frac{\mu}{2}}}{\Gamma(-\mu-\nu)\,\Gamma(\nu+1)}\int_0^\infty (\zeta+\cosh t)^{\mu-\nu-1}(\sinh t)^{2\nu+1}\,\d t,
\end{align*}
for $\Re(-\mu)>\Re\nu>-1$ and $\zeta$ not on the real axis between $-1$ and $\infty$.
See \cite[pp. 155-156]{EMOT53}.

We also have the inverse formula of the above Fourier transform:
\begin{align*}
f(r)=C_{n}\int_0^\infty \F f(\lambda)\,\varphi_\lambda(r)\,|\bc(\lambda)|^{-2}\,\d\lambda,
\end{align*}
where $C_n=2^{n-1}\uppi^{n/2-1}\Gamma(n/2)^{-1}$ and the Harish-Chandra $\bc$-function is defined by
\begin{align*}
\bc(\lambda)={\Gamma(n-1)\over \Gamma((n-1)/2)}{\Gamma(\i\lambda)\over\Gamma(\i\lambda+(n-1)/2)},
\end{align*}
see \cite[(3.1.3 and 3.3.3)]{Bra94}.

\subsection{Spectral multipliers on the hyperbolic space}\label{sec: multiplier}
In view of \cite[Chapter 8, Proposition 5.1]{Tay23}, the non-negative Laplace-Beltrami operator $-\Delta$ on $\H^n$ has a continuous spectrum $[(n-1)^2/4,\infty)$ on $L^2(\H^n)$. Hence, the operator $D=\sqrt{-\Delta-(n-1)^2/4}$ is a non-negative self-adjoint operator on $L^2(\H^n)$, and one can define a spectral multiplier $m(D)$ via the spectral theorem, where $m$ is a Borel function on $[0,\infty)$. Furthermore, one can write
\begin{align*}
m(D)(f)(x)=\int_{\H^n}f(y)\,K(d(x,y))\,\d y,
\end{align*}
with the radial kernel
\begin{align*}
K(r):=\F^{-1}(m)(r)=\int_0^\infty m(\lambda)\,\varphi_\lambda(r)\,|\bc(\lambda)|^{-2}\,\d\lambda.
\end{align*}
In addition, by the spectral theorem, the following Plancherel identity holds:
\begin{align}\label{Eq:2.3}
\|m(D)(f)\|_{L^2(\H^n)}
=
\|m\, (\F f)\|_{L^2(\R_+ \times \SS^{n-1},\,C_{n} |\bc(\lambda)|^{-2}\,\d\lambda)}
\end{align}
see, for instance, \cite[Theorem 3.3.2]{Bra94}. For simplicity, we will denote the $L^p$-norm $\|\cdot\|_{L^p}$ by $\|\cdot\|_{p}$ for all $1\le p \le \infty$.

In our study, we consider the family of averaging operators $\lbrace{M_{t}^{\alpha}}\rbrace_{\alpha\in\C}$ introduced in \cite{ElK80}, where $M_{t}^{\alpha}$ is given by
\begin{align}\label{Eq:2.4}
M^\alpha_t(f)(x)
= \frac{1}{\Gamma(\alpha)}
\frac{2\e^t}{(\e^t-1)^{2\alpha}} (\sinh t)^{2-n}
\int_{B(x,t)}[\e^t x-y]^{\alpha-1}f(y)\,\d y,
\end{align}
Here $B(x,t)$ is a geodesic ball of radius $t$ centered at $x\in\H^n$. According to \cite[Lemma 2.3]{CSWY24}, the operator $M_{t}^{\alpha}$ can be analytically continued to $\alpha \in \C$ with $\Re\alpha>(1-n)/2$, satisfying
\begin{align*}
M^\alpha_t(f)(x)=m^\alpha_t(D)(f)(x),
\end{align*} 
where
\begin{align}\label{Eq:2.6}
m^\alpha_t(\lambda)=
2^{\frac{n-2}{2}+\alpha} \Gamma\!\left(\frac{n}{2}\right)
\frac{\e^{\alpha t}}{(\e^t-1)^{2\alpha}}
(\sinh t)^{\frac{2-n}{2}+\alpha}
P_{-\frac{1}{2}+\i\lambda}^{\frac{2-n}{2}-\alpha}(\cosh t).
\end{align}
In other words, the operator $M^\alpha_t$ is a Fourier multiplier with symbol $m^\alpha_t$. Note that $m^0_t(\lambda)$ corresponds to the spherical function $\varphi_{\lambda}(t)$ defined by \eqref{def: phi P}.

The following estimates for $m^\alpha_t$ will be used in the proof of the main theorem.
\begin{lemma}\label{le3}
Suppose $\alpha\in\C$ with $\Re\alpha>(1-n)/2$. Then, the following estimates hold for all $\lambda\in\R$.
\begin{enumerate}[label=(\roman*)]
\item For all $t>0$, we have
\begin{align}\label{Eq:2.9}
|m^\alpha_t(\lambda)| \lesssim \,(1+t)\,\e^{-\frac{n-1}{2} t}.
\end{align}
\item If $0<t\leq 1$, we have
\begin{align}\label{Eq:2.7}
\left|\frac{\d}{\d \lambda} m^\alpha_t(\lambda)\right|
\lesssim t.
\end{align}
If in addition $|\lambda|t\geq1$, then
\begin{align}\label{Eq:2.8}
|m^\alpha_t(\lambda)| \lesssim (|\lambda|t)^{-(\Re\alpha+\frac{n-1}{2})}.
\end{align}
\end{enumerate}
\end{lemma}

\begin{proof}
Estimates \eqref{Eq:2.9} and \eqref{Eq:2.8} follow directly from \cite[Lemma 2.4]{CSWY24}. To prove \eqref{Eq:2.7}, we use the following integral representation of the Legendre function:
\begin{align*}
P_{-\frac{1}{2}+\i\lambda}^{\frac{2-n}{2}-\alpha}(\cosh t)
= C_{\alpha} (\sinh t)^{\frac{2-n}{2}-\alpha}
\int_0^t(\cosh t-\cosh s)^{\alpha+\frac{n-3}{2}}\cos (\lambda s)\,\d s,
\end{align*}
see \cite[p. 156, (8)]{EMOT53}. Substituting this formula into \eqref{Eq:2.6} and differentiating both sides with respect to $\lambda$, we obtain, for $0<t\leq 1$, 
\begin{align*}
\left|{\d\over\d \lambda}m^\alpha_t(\lambda)\right|
\lesssim t^{-2\Re\alpha-n+2}
\int_0^t(\cosh t-\cosh s)^{\Re\alpha+\frac{n-3}{2}}s\,\d s.
\end{align*}

Note that
\begin{align*}
\cosh t-\cosh s
= {\e^t\over 2}(1-\e^{-t-s})(1-\e^{-t+s})
\sim t^2-s^2,
\end{align*}
for all $0\leq s \leq t\leq 1$. Therefore,  we obtain, for $0<t\leq 1$,
\begin{align*}
\left|{\d\over\d \lambda}m^\alpha_t(\lambda)\right|
\lesssim t^{-2\Re\alpha-n+2} 
\int_0^t(t^2-s^2)^{\Re\alpha+\frac{n-3}{2}} s\,\d s 
\lesssim t,
\end{align*}
which is \eqref{Eq:2.7}. The proof is therefore complete.
\end{proof}

Finally, recalling \eqref{Eq:2.4} and formula
\begin{align*}
[\e^tx-y]=2\e^t(\cosh t-\cosh d(x,y)),
\end{align*}
we see that the Fourier multiplier $M^\alpha_t=m^\alpha_t(D)$ has the radial kernel
\begin{align}\label{Eq:2.5}
K^\alpha_t(r)
=\frac{1}{\Gamma(\alpha)}
\frac{(2\e^t)^\alpha}{(\e^t-1)^{2\alpha}}
(\sinh t)^{2-n}\,(\cosh t-\cosh r)^{\alpha-1}\idct_{[0,t]}(r)
\end{align}
provided that $\Re\alpha>0$.
\section{Proof of Theorem~\ref{thm: main}}\label{Sec:3}
For $\alpha\in\C$ with $\Re\alpha>(1-n)/2$, we introduce the lacunary maximal operator
\begin{align}\label{ennn}
\L^\alpha(f)(x)=\sup_{t\in\Lambda}|m^\alpha_t(D)(f)(x)|,
\end{align}
where the lacunary set is defined by \eqref{eq:deflac}, and $m_{t}^{\alpha}(D)$ is the Fourier multiplier corresponding to \eqref{Eq:2.6}. Since
\begin{align*}
\F(\d\omega_t)(\lambda)
= \varphi_\lambda(t)
= m^0_t(\lambda),
\end{align*}
the operator $\L$ defined by \eqref{def: lac} coincides exactly with the operator $\L^0$ defined in \eqref{ennn}. Therefore, this section is devoted to proving the $L^p$-boundedness of the operator $\L^{\alpha}$, with the result stated as follows.
\begin{theorem}\label{thm: main alpha}
Suppose that $1<p\le\infty$ and $\alpha\in\C$ with
\begin{align}\label{def: alpha}
\Re\alpha >
(n-1)\left|\frac{1}{p}-\frac{1}{2}\right|-\frac{n-1}{2}.
\end{align}
Then the operator $\L^{\alpha}$ is bounded on $L^{p}(\H^n)$.
\end{theorem}

\begin{remark}\label{rmk: range}
According to \cite[Theorem 1.2]{CSWY24} and \cite[Theorem 3]{ElK80}, the spherical maximal operator $\mathcal{M}^\alpha_*(f)(x)=\sup_{t>0}|m^\alpha_t(D)(f)(x)|$ is bounded on $L^p(\H^n)$ in the following cases:
\begin{itemize}
\item $1<p\leq 2$ and $\Re\alpha>(1-n)+n/p$;
\item $2<p\leq p_n$ and $\Re\alpha>(2-n)/p-(p-2)/[pp_n(p_n-2)]$;
\item $p_n<p\leq \infty$ and $\Re\alpha>(2-n)/p-1/(pp_n)$.
\end{itemize}
Here $p_n=4$ for $n=2$ and $p_n=2(n+1)/(n-1)$ for $n\geq 3$. These relations correspond to the region above the thick dotted segments $OABC$ in Figure \ref{Fig:1}, while the relation \eqref{def: alpha} corresponds to the region above the thick dashed folded segments $ODE$. From this, we intuitively observe that the lacunary maximal operator $\mathcal{L}^{\alpha}$ enjoys much broader $L^p$-boundedness than the spherical maximal operator $\mathcal{M}^\alpha_*$.

\begin{figure}
\centering
\begin{tikzpicture}[x=1.00mm, y=1.00mm, inner xsep=0pt, inner ysep=0pt, outer xsep=0pt, outer ysep=0pt]
\path[line width=0mm] (11.26,42.80) rectangle +(83.15,77.11);
\definecolor{L}{rgb}{0,0,0}
\path[line width=0.15mm, draw=L] (19.90,44.80) -- (19.90,113.72);
\definecolor{F}{rgb}{0,0,0}
\path[line width=0.15mm, draw=L, fill=F] (19.90,113.72) -- (19.20,110.92) -- (19.90,113.72) -- (20.60,110.92) -- (19.90,113.72) -- cycle;
\path[line width=0.15mm, draw=L] (16.05,80.04) -- (87.66,80.04);
\path[line width=0.15mm, draw=L, fill=F] (87.66,80.04) -- (84.86,80.74) -- (87.66,80.04) -- (84.86,79.34) -- (87.66,80.04) -- cycle;
\draw(18.80,115.00) node[anchor=base east]{\selectfont $\Re\alpha$};
\draw(89.92,77.18) node[anchor=base west]{\selectfont $1\over p$};
\path[line width=0.15mm, draw=L] (80.01,110.01) circle (1.25mm);
\path[line width=0.15mm, draw=L] (50.00,50.01) circle (1.25mm);
\path[line width=0.15mm, draw=L] (19.90,80.00) circle (1.25mm);
\path[line width=0.15mm, draw=L] (80.01,80.04) circle (1.25mm);
\path[line width=0.15mm, draw=L] (50.00,65.08) circle (1.25mm);
\path[line width=0.15mm, draw=L] (33.61,70.46) circle (1.25mm);
\path[line width=0.60mm, draw=L, dash pattern=on 0.50mm off 0.50mm] (19.90,80.11) -- (33.65,70.50);
\path[line width=0.60mm, draw=L, dash pattern=on 0.50mm off 0.50mm] (79.44,109.97);
\path[line width=0.60mm, draw=L, dash pattern=on 0.50mm off 0.50mm] (33.69,70.50) -- (50.01,64.97);
\path[line width=0.60mm, draw=L, dash pattern=on 0.50mm off 0.50mm] (50.04,65.08) -- (80.01,110.07);
\draw(17.89,81.67) node[anchor=base east]{\selectfont $O$};
\path[line width=0.60mm, draw=L, dash pattern=on 2.00mm off 1.00mm] (19.87,80.09) -- (49.99,50.00);
\path[line width=0.60mm, draw=L, dash pattern=on 2.00mm off 1.00mm] (50.00,50.01) -- (79.91,79.95);
\path[line width=0.15mm, draw=L, dash pattern=on 2.00mm off 1.00mm] (50.00,50.07) -- (50.00,80.05);
\path[line width=0.15mm, draw=L, dash pattern=on 2.00mm off 1.00mm] (80.00,80.07) -- (80.00,110.05);
\path[line width=0.15mm, draw=L, dash pattern=on 2.00mm off 1.00mm] (33.64,70.36) -- (33.64,80.06);
\path[line width=0.15mm, draw=L, dash pattern=on 2.00mm off 1.00mm] (80.01,110.00) -- (19.87,110.01);
\path[line width=0.15mm, draw=L, dash pattern=on 2.00mm off 1.00mm] (49.99,64.95) -- (20.06,64.93);
\path[line width=0.15mm, draw=L, dash pattern=on 2.00mm off 1.00mm] (49.99,49.93) -- (19.93,49.93);
\path[line width=0.15mm, draw=L, dash pattern=on 2.00mm off 1.00mm] (33.56,70.58) -- (19.83,70.55);
\draw(49.27,67.74) node[anchor=base east]{\selectfont $B$};
\draw(51.58,46.26) node[anchor=base west]{\selectfont $D$};
\draw(81.70,110.62) node[anchor=base west]{\selectfont $C$};
\draw(49.93,84.00) node[anchor=base]{\selectfont $1\over 2$};
\draw(19.01,63.00) node[anchor=base east]{\selectfont $2-n\over 2$};
\draw(19.24,48.50) node[anchor=base east]{\selectfont $1-n\over 2$};
\draw(19.01,70.00) node[anchor=base east]{\selectfont ${2-n\over p_n}-{1\over p_n^2}$};
\draw(33.62,84.00) node[anchor=base]{\selectfont $1\over p_n$};
\draw(18.26,108.68) node[anchor=base east]{\selectfont $1$};
\draw(80.04,73.50) node[anchor=base]{\selectfont $1$};
\draw(81.00,81.26) node[anchor=base west]{\selectfont $E$};
\draw(34.28,71.94) node[anchor=base west]{\selectfont $A$};
\end{tikzpicture}
\caption{The admissible range of $(\Re\alpha,p)$ for which the operators $\mathcal{M}^\alpha_*$ and $\L^{\alpha}$ are respectively $L^p$-bounded.}
\label{Fig:1}
\end{figure}
\end{remark}

To prove Theorem \ref{thm: main alpha}, we consider two other operators:
\begin{align*}
\fg^\alpha_*(f)(x):=\sup_{j\in\PZ}|m^\alpha_j(D)(f)(x)|\quad\mbox{and}\quad \fl^\alpha_*(f)(x):=\sup_{j\in\PZ}|m^\alpha_{2^{-j}}(D)(f)(x)|.
\end{align*}
Since
\begin{align*}
\L^{\alpha}(f)(x)\leq \fg^\alpha_*(f)(x)+\fl^\alpha_*(f)(x)
\end{align*}
holds for all $f$ and for a.e. $x\in\H^n$, we will first establish the $L^p$-boundedness of the operators $\fg^\alpha_*$ and $\fl^\alpha_*$ separately and then derive Theorem \ref{thm: main} using an interpolation argument. We start with the nonlocal part, where the approach differs from the Euclidean setting. One main tool is the following Kunze-Stein phenomenon.
\begin{lemma}\label{Lem:2.4}
Suppose that $1<p<2$ and $\kappa$ is a non-negative radial function on $\H^n$. Then we have, for all $f\in L^p(\H^n)$,
\begin{align*}
\left\|f*\kappa\right\|_p
\lesssim 
\left(\int_0^\infty\e^{-\frac{n-1}{p'}r}\kappa(r)\,\sinh^{n-1} r\,\d r\right)\|f\|_p.
\end{align*}
\end{lemma}

\begin{proof}
Since $\kappa$ is non-negative, we know from Herz's ``rincipe de majoration'' \cite{Her70,Cow97} that 
\begin{align*}
\left\|f*\kappa\right\|_{q}
= \|f\|_{q} \int_{\H^n}
\varphi_{\i(n-1)(\frac{1}{2} - \frac{1}{q})}(x) \,\kappa(x) \df{x},
\end{align*}
where 
\begin{align*}
\left|\varphi_{\i(n-1)(\frac{1}{2} - \frac{1}{q})}(x)\right|
\lesssim \e^{-\frac{n-1}{q}|x|}
\end{align*}
for all $2<q<\infty$. See \cite[Lemma 3.4]{NPP14}, also \cite[(2.8) and (2.9)]{KRZ23}. Therefore, we obtain
\begin{align*}
\left| \int_{\H^n}
\varphi_{\i(n-1)(\frac{1}{2} - \frac{1}{q})}(x) \kappa(x) \df{x} \right|
\lesssim
\int_0^\infty\e^{-\frac{(n-1)r}{q}}\kappa(r)\,\sinh^{n-1} r\,\d r,
\end{align*}
since $\kappa$ is radial on $\H^n$. The lemma follows from a duality argument with $q=p'$.
\end{proof}

The following proposition establishes the $L^p$-boundedness of the operator $\fg^\alpha_*$ for $1<p\le2$.
\begin{proposition}\label{prop:3.2}
Suppose that $\alpha\in \C$ and $1<p<2$. Then the estimate
\begin{align}\label{est: global2}
\|\fg^\alpha_*(f)\|_2 \lesssim \|f\|_2,
\end{align}
holds provided that $\Re\alpha>(1-n)/2$; and
\begin{align}\label{est: globalp}
\|\fg^\alpha_*(f)\|_p \lesssim \|f\|_p,
\end{align}
holds provided that $\Re\alpha>0$.
\end{proposition}

\begin{remark}
It is noteworthy that an exponentially strong maximal inequality for the operator $\fg^0_*$ was previously established by Nevo in a more general setting. See, for instance, \cite[Theorem 2]{Nev98}.
\end{remark}

\begin{proof}
According to the Plancherel identity \eqref{Eq:2.3}, we have
\begin{align*}
\|\fg^\alpha_*(f)\|^2_2&\leq \sum_{j=1}^\infty\|m^\alpha_j(D)(f)\|_2^2
\leq \sum_{j=1}^\infty\|m^\alpha_j\|^2_\infty\|f\|_2^2.
\end{align*}
We deduce from the estimate \eqref{Eq:2.9} that 
\begin{align*}
\|\fg^\alpha_*(f)\|^2_2&\leq \sum_{j=1}^\infty\|m^\alpha_j(D)(f)\|_2^2
\lesssim \left(\sum_{j=1}^\infty j^2\e^{-(n-1)j}\right)\|f\|_2^2
\lesssim \|f\|_2^2.
\end{align*}
We then obtain the $L^2$-estimate \eqref{est: global2}.

To prove the $L^p$-estimate \eqref{est: globalp}, we need the kernel of the Fourier multiplier $m^\alpha_t(D)$. We deduce from \eqref{Eq:2.5} and
\begin{align}\label{Eq:2.500001}
\cosh t-\cosh r={\e^t\over 2}(1-\e^{r-t})(1-\e^{-r-t})
\end{align}
that, for all $t\ge1$,
\begin{align*}
|K^\alpha_t(r)|
&\lesssim \e^{-(n-1)t}(1-\e^{r-t})^{\Re\alpha-1}\idct_{[0,t]}(r)\\
&\lesssim
\underbrace{\e^{-(n-1)(t-\frac12)}\idct_{[0,t-\frac12]}(r)}_{
K^{\alpha,1}_t(r)}
+
\underbrace{\e^{-(n-1)t}(t-r)^{\Re\alpha-1}\idct_{[t-\frac12,t]}(r)}_{
K^{\alpha,2}_t(r)}.
\end{align*}
Denote by 
\begin{align*}
\fg^{\alpha,\ell}_*(f)(x):=\sup_{j\in\PZ}\left|\int_{\H^n}f(y)\,K^{\alpha,\ell}_j(d(x,y))\,d y\right| \quad \mbox{for }\ell=1,2.
\end{align*}
Then we have clearly 
\begin{align*}
\|\fg^\alpha_*(f)\|_p\leq \|\fg^{\alpha,1}_*(|f|)\|_p+\|\fg^{\alpha,2}_*(|f|)\|_p.
\end{align*}

For the $L^p$-boundedness of the operator $\fg^{\alpha,1}_*$, note that $\fg^{\alpha,1}_*(|f|)(x)$ is dominated (up to a constant) by the Hardy-Littlewood maximal function
\begin{align*}
f^*(x):=\sup_{t>0}{1\over |B(x,t)|}\int_{B(x,t)}|f(y)|\,\d y
\end{align*}
of $f$ for a.e. $x\in\H^n$. We know from \cite[Theorem 2]{CS74} that $\|f^*\|_p \lesssim \|f\|_p$ for all $1<p\leq \infty$, thus $\fg^{\alpha,1}_*$ is bounded on $L^p(\H^n)$ for such $p$.

It remains to show that $\fg^{\alpha,2}_*$ is bounded on $L^p(\H^n)$ for $1<p<2$, which is a consequence of Lemma \ref{Lem:2.4}: 
\begin{align*}
&\|\fg^{\alpha,2}_*(f)\|^p_p \\
&\leq \sum_{j=1}^\infty \left\|\int_{\H^n}f(y)\,K^{\alpha,2}_j(d(x,y))\,\d y\right\|_p^p\\
&\lesssim \left[\sum_{j=1}^\infty\left(\int_{j-\frac12}^j \e^{-\frac{n-1}{p'}r}\e^{-(n-1)j}(j-r)^{\Re\alpha-1}\sinh^{n-1} r\,\d r\right)^p\right]\|f\|^p_p\\
&\lesssim \left(\sum_{j=1}^\infty\e^{-(n-1)(p-1)j}\right)\|f\|^p_p
\lesssim \|f\|_p^p
\end{align*}
provided that $\Re\alpha>0$ and $1<p<2$. This completes the proof of Proposition \ref{est: globalp}.
\end{proof}

Now we turn to the local part. We will show that the operator $\fl^\alpha_{*}$ defined by
\begin{align*}
\fl^\alpha_*(f)(x)=\sup_{j\in\PZ}|m^\alpha_{2^{-j}}(D)(f)(x)|
\end{align*}
is bounded from $L^2(\H^n)$ to $L^2(\H^n)$ and from $L^1(\H^n)$ to $L^{1,\infty}(\H^n)$, respectively. The argument relies on the estimates of the symbol \eqref{Eq:2.6} and follows the spirit of the arguments presented in \cite{Cal79}.

\begin{proposition}\label{prop:3.3}
Suppose $\alpha\in \C$ with $\Re\alpha>(1-n)/2$. Then
\begin{align*}
\|\fl^\alpha_{*}(f)\|_2\lesssim\|f\|_2.
\end{align*}
\end{proposition}

\begin{proof}
For $t\ge0$, we denote by $\cH_t=\e^{tD^2}$ the heat diffusion semigroup on $\H^n$, where $D^2$ is the shifted Laplacian defined in Section \ref{sec: multiplier}. Let $\cH_*$ be the heat maximal operator defined by
\begin{align*}
\cH_*(f)(x)=\sup_{t>0}|\cH_t(f)(x)|.
\end{align*}

Therefore, the operator $\fl^\alpha_{*}$ satisfies
\begin{align*}
\|\fl^\alpha_{*}(f)\|_2\leq 
\underbrace{\sup_{t>0}|m^\alpha_t(0)|\,\|\cH_*(f)\|_2 \vphantom{\left\|\sup_{j\in \PZ}|T_j(f)|\right\|_2}}_{\mathcal{I}_1}
+
\underbrace{\left\|\sup_{j\in \PZ}|T_j(f)|\right\|_2}_{\mathcal{I}_2},
\end{align*}
where
\begin{align*}
T_j(f)(x)=(m^\alpha_{2^{-j}}(D)-m^\alpha_{2^{-j}}(0)\cH_{2^{-2j}})(f)(x).
\end{align*}
On the one hand, we know from \eqref{Eq:2.9} that $m^\alpha_t(0)$ is bounded for all $t > 0$. On the other hand, the heat maximal operator $\cH_*$ is known to be bounded on $L^2(\H^n)$ (see \cite{Ste70}). Combining these two facts, we obtain $\mathcal{I}_1\lesssim \|f\|_2$.

Regarding $\mathcal{I}_2$, we deduce from Plancherel's identity \eqref{Eq:2.3} that
\begin{align*}
&\|T_j(f)\|_2^2 \\
&= C_n \int_{0}^{\infty}\int_{\SS^{n-1}}
\left|m^\alpha_{2^{-j}}(\lambda)
-m^\alpha_{2^{-j}}(0)\,\e^{-2^{-2j}\lambda^2}\right|^2\,
|(\F f)(\lambda,\sigma)|^{2} |\bc(\lambda)|^{-2}
\df{\sigma} \df{\lambda} \\
&\le \left\lbrace\sup_{\lambda\in\R_+}
\left|m^\alpha_{2^{-j}}(\lambda)
-m^\alpha_{2^{-j}}(0)\,\e^{-2^{-2j}\lambda^2}\right|^2 \right\rbrace
\|\F f\|_{L^2(\R_+ \times \SS^{n-1},\,C_{n} |\bc(\lambda)|^{-2}\,\d\lambda)}^2,
\end{align*}
for every $j\in\Z_{+}$, then
\begin{align*}
\mathcal{I}_2^2 \leq \sum_{j=1}^\infty\|T_j(f)\|_2^2
\leq 
\underbrace{
\left\lbrace\sup_{\lambda\in\R_+} \sum_{j=1}^\infty
\left|m^\alpha_{2^{-j}}(\lambda)
-m^\alpha_{2^{-j}}(0)\,\e^{-2^{-2j}\lambda^2}\right|^2 \right\rbrace
}_{\mathcal{I}_3}
\|f\|_2^2.
\end{align*}

To show the boundedness of $\mathcal{I}_3$, we divide the argument into two cases depending on whether $2^{-j}\lambda>1$ or $2^{-j}\lambda\le1$. On the one hand, we deduce straightforwardly from \eqref{Eq:2.9} and \eqref{Eq:2.8} that 
\begin{align*}
&\sum_{\lbrace{j\in\Z_+ \mid 2\leq 2^j<\lambda}\rbrace}
\left|m^\alpha_{2^{-j}}(\lambda)-m^\alpha_{2^{-j}}(0)\,\e^{-2^{-2j}\lambda^2}\right|^2\\
&\lesssim 
\sum_{\lbrace{j\in\Z_+ \mid 2\leq 2^j<\lambda}\rbrace}
\left(|m^\alpha_{2^{-j}}(\lambda)|^2+\left|m^\alpha_{2^{-j}}(0)\,\e^{-2^{-2j}\lambda^2}\right|^2\right)\\
&\lesssim 
\sum_{\lbrace{j\in\Z_+ \mid 2\leq 2^j<\lambda}\rbrace}
(2^{-j}\lambda)^{-2\left(\Re\alpha+\frac{n-1}{2}\right)},
\end{align*}
which is finite for all $\lambda>0$ provided that $\Re\alpha>(1-n)/2$. On the other hand, by applying the mean value theorem and combining \eqref{Eq:2.9} and \eqref{Eq:2.7}, we obtain
\begin{align*}
&\sum_{\lbrace{j\in\Z_+ \mid 2^j\ge\lambda}\rbrace}
\left|m^\alpha_{2^{-j}}(\lambda)-m^\alpha_{2^{-j}}(0)\,\e^{-2^{-2j}\lambda^2}\right|^2\\
&\lesssim
\sum_{\lbrace{j\in\Z_+ \mid 2^j\ge\lambda}\rbrace}
\left(\left|m^\alpha_{2^{-j}}(\lambda)-m^\alpha_{2^{-j}}(0)\right|^2+\left|m^\alpha_{2^{-j}}(0)\,\left(\e^{-2^{-2j}\lambda^2}-1\right)\right|^2\right)\\
&\lesssim
\sum_{\lbrace{j\in\Z_+ \mid 2^j\ge\lambda}\rbrace}
(2^{-j}\lambda)^2,
\end{align*}
which is also finite for all $\lambda>0$. We conclude that $\mathcal{I}_3$ is bounded and then $|\mathcal{I}_2|\lesssim\|f\|_2$, which completes the proof.
\end{proof}

The final component needed to prove Theorem \ref{thm: main alpha} is stated as follows.
\begin{proposition}\label{prop:3.4}
Suppose $\alpha\in \C$ with $\Re\alpha>0$. Then the operator $\fl^\alpha_*$ is bounded from $L^1(\H^n)$ to $L^{1,\infty}(\H^n)$. In other words, the estimate 
\begin{align}\label{Eq:3.1}
|\{x\in\H^n:\fl^\alpha_*(f)(x)>\gamma\}|\lesssim {\|f\|_1\over\gamma},
\end{align}
holds for all $\gamma>0$.
\end{proposition}

\begin{proof}
Recall that for $\Re\alpha>0$, the Fourier multiplier $m^\alpha_t(D)=M^\alpha_t$ has a radial kernel $K^\alpha_t(r)$ given by \eqref{Eq:2.5}. It follows from \eqref{Eq:2.5} and \eqref{Eq:2.500001} that, for $0<t\leq 1$,
\begin{align}\label{Eq:3.2}
|K^\alpha_t(r)|\lesssim \widetilde{K}^\alpha_t(r):=t^{-n}\left(1-{r\over t}\right)^{\Re\alpha-1}\idct_{[0,1]}\!\left({r\over t}\right).
\end{align}
We denote by $\widetilde{\fl}^\alpha_*$ the maximal operator associated with this positive radial kernel:
\begin{align*}
\widetilde{\fl}^\alpha_*(f)(x)=\sup_{j\in\PZ}\left|\int_{\H^n}f(y)\,\widetilde{K}^\alpha_{2^{-j}}(d(x,y))\,\d y\right|.
\end{align*}
Therefore, to prove \eqref{Eq:3.1}, it suffices to show
\begin{align}\label{Eq:3.3}
|\{x\in\H^n:\widetilde{\fl}^\alpha_*(f)(x)>\gamma\}|\lesssim{\|f\|_1\over\gamma}
\end{align}
for all non-negative function $f\in L^1(\H^n)$.

\textbf{Proof of \eqref{Eq:3.3}.}
Let us localize $\widetilde{\fl}^\alpha_*$ using the family $\{B_k\}$ of balls with radius $1$ satisfying $\|\sum_k\idct_{2B_k}\|_\infty<\infty$, where $2B_k$ denotes the ball of the same center as $B_k$'s and twice the radius of $B_k$. We claim that, for each $k$
\begin{align}\label{Eq:3.4}
|\{x\in B_k:\widetilde{\fl}^\alpha_*(\idct_{2B_k}f)(x)>\gamma\}|\lesssim {\|\idct_{2B_k}f\|_1\over \gamma},
\end{align}
for all non-negative $f\in L^1(\H^n)$. From the definition \eqref{Eq:3.2} of the kernel $\widetilde{K}^\alpha_t$, we see that, for $0<t\le1$ and $x\in B_k$, $\widetilde{K}^\alpha_t(d(x,\cdot))$ is supported in $B(x,1)\subset 2B_k$ for any $x\in\H^n$. Hence, if \eqref{Eq:3.4} holds, it implies that, for all $\gamma>0$,
\begin{align*}
\mbox{LHS of \eqref{Eq:3.3}}&\lesssim \sum_k|\{x\in B_k:\widetilde{\fl}^\alpha_*(\idct_{2B_k}f)(x)>\gamma\}\\
&\lesssim  {1\over\gamma}\sum_k\|\idct_{2B_k}f\|_1\\
&\lesssim {\|f\|_1\over\gamma}=\mbox{RHS of \eqref{Eq:3.3}},
\end{align*}
which proves the proposition.

\textbf{Proof of \eqref{Eq:3.4}.}
Now, let us fix a $k$ and prove the estimate \eqref{Eq:3.4}. Observe that \eqref{Eq:3.4} holds trivially for $\gamma\leq \|\idct_{2B_k}f\|_1/|B_k|$. Hence, in the following, we assume $\gamma>\|\idct_{2B_k}f\|_1/|B_k|$. Note that the function $\idct_{2B_k}f$ is local and $\H^n$ is local doubling, i.e., for all $x\in\H^n$ and $0<r\leq 4$ we have $|B(x,2r)|\lesssim |B(x,r)|$. Therefore we have the Calder\'{o}n-Zygmund decomposition for $\idct_{2B_k}f$ at height $\delta>\|\idct_{2B_k}f\|_1/|2B_k|$, that is, $\idct_{2B_k}f$ can be decomposed into $g+b$, where $0\leq g\lesssim \delta$ a.e. and $b=\sum_lb_l$ for
\begin{align}\label{Eq: bad_function}
b_l(x)=\left(f(x)-{1\over|Q_l|}\int_{Q_l}f(y)\,\d y\right)\idct_{Q_l}(x).
\end{align}
Here $\{Q_l\}$ is a family of balls contained in $4B_k$ and satisfying the following properties:
\begin{align}\label{Eq:3.5}
\delta\leq {1\over|Q_l|}\int_{Q_l}f(y)\,\d y\lesssim\delta,\quad\left\|\sum_l\idct_{2Q_l}\right\|_\infty<\infty,\quad\mbox{and}\quad\left|\bigcup_lQ_l\right|\lesssim {\|\idct_{2B_k}f\|_1\over\delta}.
\end{align}
For details see \cite[Corollary (2.3)]{CW71}.

It follows from \eqref{Eq:3.2} along with \eqref{Eq:2.1} that $\widetilde{K}^\alpha_t\in L^1(\H^n)$ for $\Re\alpha>0$ and $0<t\leq 1$. This, together with the fact that $0\leq g\lesssim\delta$, we have
\begin{align*}
\left|\left\{x\in B_k:\widetilde{\fl}^\alpha_*(g)(x)>\frac{\gamma}{2}\right\}\right|=0,
\end{align*}
by taking $\gamma=C\delta$ with $C$ (independent of $k$ and $f$) suitably large.

Hence, by using Chebyshev's inequality, we obtain
\begin{align}\label{Eq:3.6}
\mbox{LHS of \eqref{Eq:3.4}}\leq \left|\bigcup_l2Q_l\right|+{2\over\gamma}\sum_l\int_{B_k\setminus 2Q_l}\widetilde{\fl}^\alpha_*(b_l)(x)\,\d x,
\end{align}
where
\begin{align}\label{Eq:3.7}
\left|\bigcup_l2Q_l\right|\lesssim {\|\idct_{2B_k}f\|_1\over\gamma},
\end{align}
according to \eqref{Eq:3.5}.

Regarding the second term on the right-hand side of \eqref{Eq:3.6}, note that by \eqref{Eq: bad_function} $b_l$ have mean value $0$, i.e., $\int b_l=0$, and therefore
\begin{align*}
&\sum_l\int_{B_k\setminus 2Q_l}\widetilde{\fl}^\alpha_*(b_l)(x)\,\d x\\
\leq &\sum_l\int_{\H^n\setminus 2Q_l}\sum_{j=1}^\infty\left|\int_{Q_l}\widetilde{K}^\alpha_{2^{-j}}(d(x,y))\,b_l(y)\,\d y\right|\,\d x\\
\leq &\sum_l\int_{Q_l}
\Bigg(\sum_{j=1}^\infty
\underbrace{
\int_{\H^n\setminus 2Q_l}|\widetilde{K}^\alpha_{2^{-j}}(d(x,y))-\widetilde{K}^\alpha_{2^{-j}}(d(x,y_l))|\,\d x}_{\mathcal{I}_{j,l}(y)}
\Bigg) |b_l(y)|\,\d y
\end{align*}
where $y_l$ denotes the center of $Q_l$. We claim that, for each $l$,
\begin{align}\label{Eq:3.9}
\sum_{j=1}^\infty \mathcal{I}_{j,l}(y)=\mathrm{O}(1)
\qquad\forall y\in Q_l,
\end{align}
which implies that
\begin{align}\label{Eq:3.8}
\sum_l\int_{B_k\setminus 2Q_l}\widetilde{\fl}^\alpha_*(b_l)(x)\,\d x
\leq 
\sum_l\int_{Q_l}
\left(\sum_{j=1}^\infty \mathcal{I}_{j,l}(y)\right)
|b_l(y)|\,\d y
\lesssim
\|\idct_{2B_k}f\|_1.
\end{align}

We obtain \eqref{Eq:3.4} by substituting the estimates \eqref{Eq:3.7} and \eqref{Eq:3.8} back into \eqref{Eq:3.6}.

\textbf{Proof of \eqref{Eq:3.9}.}
The final step is to prove the claim \eqref{Eq:3.9} regarding the estimate of $\mathcal{I}_{j,l}$. First, note from the kernel expression \eqref{Eq:3.2} that if $2^jr_l>1$, where $r_l$ is the radius of $Q_l$, then
\begin{align*}
|\widetilde{K}^\alpha_{2^{-j}}(d(x,y))-\widetilde{K}^\alpha_{2^{-j}}(d(x,y_l))|=0\quad \mbox{for }x\in(2Q_l)^c,\ y\in Q_l.
\end{align*}
Hence, we can restrict the proof to the case where $2^jr_l\le1$ and show that, for each $l$,
\begin{align}\label{Eq:3.10}
\sum_{\lbrace{j\in\Z_+\mid2\leq 2^j\leq {r_l}^{-1}}\rbrace}
\mathcal{I}_{j,l}(y) =\mathrm{O}(1)
\qquad \forall y\in Q_l.
\end{align}

To this end, first observe that if $2^{-100}<2^jr_j\leq 1$, we deduce directly from \eqref{Eq:2.1} and \eqref{Eq:3.2} that
\begin{align*}
\mathcal{I}_{j,l}(y)
&\leq 2\int_0^{2^{-j}}2^{nj}r^{n-1}(1-2^jr)^{\Re\alpha-1}\,\d r\\
&= 2\int_0^{1} r^{n-1}(1-r)^{\Re\alpha-1}\,\d r,
\end{align*}
which is finite provided that $\Re\alpha>0$. 

Otherwise, if $2^jr_l\leq 2^{-100}$, we decompose the integral $\mathcal{I}_{j,l}(y)$ into 
\begin{align}\label{Eq:3.12}
\mathcal{I}_{j,l}(y) &=
\underbrace{\int_{E_{j,l}(y)}|\widetilde{K}^\alpha_{2^{-j}}(d(x,y))-\widetilde{K}^\alpha_{2^{-j}}(d(x,y_l))|\,\d x}_{\mathcal{J}_1}\notag\\
&\quad+\underbrace{\int_{\H^n\setminus (E_{j,l}(y)\cup2Q_l)}|\widetilde{K}^\alpha_{2^{-j}}(d(x,y))-\widetilde{K}^\alpha_{2^{-j}}(d(x,y_l))|\,\d x}_{\mathcal{J}_2}
\end{align}
where the set $E_{j,l}(y)$ is defined by 
\begin{align*}
E_{j,l}(y)=(B(y,2^{-j})\setminus B(y,2^{-j}-2r_l))\bigcup (B(y_l,2^{-j})\setminus B(y_l,2^{-j}-2r_l)).
\end{align*}

By using \eqref{Eq:2.1} and \eqref{Eq:3.2} again, we know, on the one hand, that
\begin{align}\label{Eq:3.13}
\mathcal{J}_1&\leq 2\sup_{y'\in Q_l}\int_{B(y',2^{-j})\setminus B(y',2^{-j}-3r_l)}|\widetilde{K}^\alpha_{2^{-j}}(d(x,y'))|\,\d x \notag\\
&\leq 2\int_{2^{-j}-3r_l}^{2^{-j}}2^{nj} r^{n-1}(1-2^jr)^{\Re\alpha-1}\,\d r\lesssim  (2^jr_l)^{\Re\alpha},
\end{align}
provided that $\Re\alpha>0$. On the other hand, we have
\begin{align*}
\mathcal{J}_2
&\leq \int_{\widetilde{Q}}|\widetilde{K}^\alpha_{2^{-j}}(d(x,y))-\widetilde{K}^\alpha_{2^{-j}}(d(x,y_l))|\,\d x\\
&\leq \int_0^{d(y_l,y)}\left(\int_{\widetilde{Q}}\left|\langle(\nabla\widetilde{K}^\alpha_{2^{-j}}(d(x,\cdot)))(\gamma_{y,y_l}(\theta)),\dot{\gamma}_{y,y_l}(\theta)\rangle_{\gamma_{y,y_l}(\theta)}\right|\,\d x\right)\d\theta
\end{align*}
where $\widetilde{Q}$ denotes the set $B(y_l,2^{-j}-2r_l)\setminus 2Q_l$. Here, we have used the mean value theorem, where $\nabla$ denotes the gradient induced by the canonical Riemannian metric $\langle\cdot,\cdot\rangle_x$ on $\H^n$, and $\gamma_{x,y}$ is the arc-length parameterized geodesic in $\H^n$ with $\gamma_{x,y}(0)=x$ and $\gamma_{x,y}(d(x,y))=y$.

By taking the supremum of $y'$ over $Q_l$, we finally obtain
\begin{align}\label{Eq:3.14}
\mathcal{J}_2
&\leq r_l \,\sup_{y'\in Q_l}\int_{B(y',2^{-j}-r_l)\setminus B(y',r_l)}\left|
\left(\widetilde{K}^\alpha_{2^{-j}}\right)'\!(d(x,y'))
\right|\d x\notag\\
&\lesssim 2^jr_l\int_{r_l}^{2^{-j}-r_l}2^{nj}r^{n-1}(1-2^jr)^{\Re\alpha-2}\,\d r\notag\\
&\lesssim\begin{cases}
2^jr_l & \mbox{for }\Re\alpha>1,\\
-2^jr_l\log(2^jr_l) & \mbox{for }\Re\alpha=1,\\
(2^jr_l)^{\Re\alpha} & \mbox{for }0<\Re\alpha<1,
\end{cases}
\end{align}
where $2^jr_l \le 2^{-100}$ as assumed. We obtain \eqref{Eq:3.10} by combining the estimates \eqref{Eq:3.12}, \eqref{Eq:3.13}, and \eqref{Eq:3.14}. The proof is now complete.
\end{proof}
\smallskip

We now arrive at the proof of Theorem \ref{thm: main alpha}, which relies on an interpolation argument based on the boundedness results of the operators $\fl^{\alpha}_*$ and $\fg^{\alpha}_*$ obtained above.

\begin{proof}[Proof of Theorem~\ref{thm: main alpha}]
On the one hand, using Marcinkiewicz's interpolation theorem, we deduce from Propositions \ref{prop:3.3} and \ref{prop:3.4} that the operator $\fl^{\alpha_0}_*$ is bounded on $L^{p_0}(\H^n)$ for $\alpha_0\in \C$ with $\Re\alpha_0>0$ and $1<p_0<2$. On the other hand, we know from \eqref{est: globalp} that the operator $\fg^{\alpha_0}_*$ is also bounded on $L^{p_0}(\H^n)$ for the same $\alpha_0$ and $p_0$. Therefore, we have
\begin{align}\label{est: alpha0}
\|\L^{\alpha_0}f\|_{p_0}\lesssim\|f\|_{p_0}
\qquad\textnormal{for $1<p_0<2$ and $\Re\alpha_0>0$}.
\end{align}

Similarly, by combining \eqref{est: global2} and Proposition \ref{prop:3.3}, we also obtain
\begin{align}\label{est: alpha1}
\|\L^{\alpha_1}f\|_{2}\lesssim\|f\|_{2}
\qquad\textnormal{for $\Re\alpha_1>\frac{1-n}{2}$}.
\end{align}

By applying interpolation to the analytic family of operators $\lbrace{\L^{\alpha}}\rbrace_{\alpha\in\C}$, we deduce from \eqref{est: alpha0} and \eqref{est: alpha1} that the operator $\L^{\alpha}$ is bounded on $L^{p}(\H^n)$ for $p_0\le p \le 2$ and 
\begin{align*}
\Re\alpha=\Re\alpha_0{\frac{1}{p}-\frac{1}{2}\over\frac{1}{p_0}-\frac{1}{2}}+\Re\alpha_1{\frac{1}{p_0}-\frac{1}{p}\over\frac{1}{p_0}-\frac{1}{2}},
\end{align*}
where $\Re\alpha_0>0$, $\Re\alpha_1>(1-n)/2$, and $1<p_0<2$ are arbitrary parameters. 

It further follows that $\L^{\alpha}$ is bounded on $L^{p}(\H^n)$ for $1<p\le 2$ and
\begin{align*}
\Re\alpha> 1-n + \frac{n-1}{p}.
\end{align*}

Also, note that Schur's test implies that $\L^{\alpha}$ is bounded on $L^\infty(\H^n)$ provided that $\Re\alpha>0$. We obtain, by using a similar interpolation argument, that the operator $\L^{\alpha}$ is bounded on $L^p(\H^n)$ for $2\le p< \infty$, provided that $\Re\alpha>(1-n)/p$. This completes the proof.
\end{proof}

This article concludes with the proof of Theorem \ref{thm: main}.
\begin{proof}[Proof of Theorem \ref{thm: main}]
Theorem \ref{thm: main} is a special case of Theorem \ref{thm: main alpha} with $\alpha=0$.
\end{proof}
\smallskip

\textbf{Acknowledgements.} The authors are greatly indebted to Lixin Yan for valuable discussions, and would also like to thank Amos Nevo for pointing out the exponentially-strong maximal inequality of the operator $\fg^0_*$ defined in Section \ref{Sec:3}. The authors are also grateful to the anonymous reviewers for their careful reading and helpful suggestions, which have improved the presentation of the paper. Y. Wang was supported by National Key R\&D Program of China 2022YFA1005700 and by NNSF of China (No. 12571111). H.-W. Zhang receives funding from the Deutsche Forschungsgemeinschaft (DFG) through SFB-TRR 358/1 2023 ---Project 491392403.


\bibliographystyle{amsplain-modified}
\bibliography{2025-WZ}

\vspace{20pt}\address{
\noindent\textsc{Yunxiang Wang:}
\href{mailto: wangyx93@mail2.sysu.edu.cn}{wangyx93@mail2.sysu.edu.cn}\\
\textsc{Department of Mathematics\\
Sun Yat-sen University\\
Guangzhou, 510275, P. R. China}}

\vspace{10pt}\address{
\noindent\textsc{Hong-Wei Zhang:}
\href{mailto:zhw.dimn@gmail.com}
{zhw.dimn@gmail.com}\\
\textsc{Institute for Advanced Study in Mathematics\\
Harbin Institute of Technology\\
Harbin, 150001, China}\\[5pt]
and\\[5pt]
\textsc{Institut f\"{u}r Mathematik\\
Universit\"{a}t Paderborn\\
Paderborn, 33098, Germany}}
\end{document}